\documentclass[11pt,reqno]{amsart}
\usepackage{amssymb,latexsym}
\usepackage{graphicx}
\usepackage{url}
  \usepackage[colorlinks=true, citecolor=blue]{hyperref}		
  \usepackage{tikz}
\usetikzlibrary{positioning,shapes,decorations,arrows.meta}

\calclayout

\makeatletter
\newcommand{\monthyear}[1]{%
  \def\@monthyear{\uppercase{#1}}}
\newcommand{\volnumber}[1]{%
  \def\@volnumber{\uppercase{#1}}}
\AtBeginDocument{%
\def\ps@plain{\ps@empty
  \def\@oddfoot{\@monthyear \hfil \thepage}%
  \def\@evenfoot{\thepage \hfil \@volnumber}}
\def\ps@firstpage{\ps@plain}
\def\ps@headings{\ps@empty
  \def\@evenhead{%
    \setTrue{runhead}%
    \def\thanks{\protect\thanks@warning}%
    \uppercase{The Fibonacci Quarterly}\hfil}%
  \def\@oddhead{%
    \setTrue{runhead}%
    \def\thanks{\protect\thanks@warning}%
    \hfill\uppercase{Comma Sequence}}%
  \let\@mkboth\markboth
  \def\@evenfoot{%
    \thepage \hfil \@volnumber}%
  \def\@oddfoot{%
    \@monthyear \hfil \thepage}%
  }%
\footskip=25pt
\pagestyle{headings}%
}
\makeatother

\theoremstyle{plain}
\numberwithin{equation}{section}
\newtheorem{thm}{Theorem}[section]
\newtheorem{theorem}[thm]{Theorem}
\newtheorem{lemma}[thm]{Lemma}

\newtheorem{conj}[thm]{Conjecture}

\begin{document}
\monthyear{Month Year}
\volnumber{Volume, Number}
\setcounter{page}{1}


 \newcommand{\seqnum}[1]{\href{https://oeis.org/#1}{\underline{#1}}}

\title{The Comma Sequence: A Simple Sequence With Bizarre Properties}

\author{\'{E}ric Angelini}
\address{32 dr\`{e}ve de Linkebeek\\
1640 Rhode-Saint-Gen\`{e}se, BELGIUM}
\email{eric.angelini@skynet.be}

\author{Michael S. Branicky}
\address{University of Kansas\\
 1520 W 15th St. \\
 Lawrence, KS 66045, USA}
\email{msb@ku.edu} 

\author{Giovanni Resta}
\address{Istituto di Informatica e Telematica del CNR\\
via Moruzzi 1\\
56124 Pisa, ITALY}      
\email{g.resta@iit.cnr.it}

\author{N. J. A. Sloane}
\address{The OEIS Foundation Inc.\\
11 South Adelaide Ave.\\
Highland Park, NJ 08904, USA}
\email{njasloane@gmail.com}

\author{David W. Wilson}
\address{3 West St., Apt. 1\\
Lebanon, NH 03766, USA}
\email{davidwwilson710@gmail.com}

\begin{abstract}
The ``comma sequence'' starts with $1$ and is defined by the property that
if $k$ and $k'$ are consecutive terms, the two-digit number formed from the last digit of $k$ and the first
digit of $k'$ is equal to the difference $k'-k$. If there is more than one such $k'$, choose the smallest,
but if there is no such $k'$ the sequence terminates.
The sequence begins $1, 12, 35, 94, 135, \ldots$, and, surprisingly, ends
at term $2137453$, which is $99999945$.
The paper analyzes the sequence and its generalizations to other starting values and other bases. A slight change in the rules allows infinitely long
comma sequences to exist.
\end{abstract}

\maketitle

\section{Introduction}\label{Sec1}

In the sequence named after the father of this journal
the next term depends on the current term and the previous term.
That is so old hat, so Twelfth Century!
In the comma sequence the next term depends on the current term  {\em and on the next term itself}, and may not even exist.

The definition says that if the current term $k$ has decimal expansion $b\ldots cd$
and the next term has decimal expansion $k' = e\ldots fg$, then $k'-k$ must
be equal to the one- or two-digit decimal number $de~ (= 10d+e)$.
We call $de$ the {\em comma-number} associated with the comma separating $k$ and $k'$,
and $k'$ the {\em comma-successor} to $k$.
The sequence starts with  $1$, and if there is a choice for $k'$, pick the smallest, and if no such $k'$ exists,
the sequence terminates.  That's  the definition.
(We give a more formal definition in \S\ref{SecCS}.)

The first few terms of the  sequence are shown in the top line  of the following table,
with the comma-numbers written directly under the commas.
In order to do this, we have exaggerated the spaces around the commas.
\begin{equation}\label{Eq1}
\begin{array}{cccccccccccccccccc}
1 &  , & 12 & ,  &  35 &  , &  94 & ,   & 135 & ,   & 186 & ,  & 248& , &331 & , & 344 & \ldots \\
~ & 11&  ~ & 23&  ~  & 59&   ~ & 41& ~    & 51 &   ~  & 62& ~   &83& ~   & 13 & ~  & \ldots
\end{array}
\end{equation}
(By the way, the reader should not miss the  excellent dramatization of the initial terms given in \cite{Vid}.)
Why is the second term of the sequence $12$?  Well, if the second term exists, its first
digit is at least $1$, so the first comma number (the number around
the first comma) is at least $11$, and so the second term is at least $1 + 11 = 12$,
and $12$ satisfies the required condition.
The second comma number is therefore at least $21$, so the third term
of the sequence is at least $12 + 21 = 33$.
However,  $33$ does not work, since the comma number would be $23$, and $12 + 23 = 35$, not $33$.
But $35$ works.  And so on. There are only ten candidates for the comma-number, so the calculations are easy.

At this point everything looks perfectly elementary and straightforward, and the reader may wonder why
we are bothering with this sequence.  The answer is that when we reach term number $2137453$,
which happens to be $99999945$, there is no choice for the next term, and the sequence ends.
This is really extraordinary.  With such a simple definition, one would expect that either the sequence
will fail after a few terms, or will continue forever.  Where did $99999945$ come from? (See
Theorem~\ref{Th1}.)

If we had taken the first term to be $3$, in fact, the sequence does die right away. The second term
must be at least $3 + 31 = 34$. Now $34$ and $35$ do not work, but $36$ does, and the
sequence begins $3,~36$. The second comma-number could only be one of $61, 62, \cdots, 69$,
but they all fail. $61$ fails because $36+61 = 97$ would produce a comma-number of $69$, not $61$.
So if we start with $3$, the comma sequence has just two terms.

The sequence (starting at 1) was contributed to the {\em On-Line Encyclopedia of Integer Sequences}
(or OEIS)  \cite{OEIS} by one of the authors (E.A.)  in 2006. It is entry \seqnum{A121805}.
At that time it was not known if the sequence was finite or infinite, but later that year
Edwin Clark determined that the sequence contains exactly $2137453$ terms.
The sequence  was mentioned in {\em Pour la Science} in 2008 \cite{PLS}.
In 2016, E.A. posted a message about the sequence to
the {\em Sequence Fans} mailing list, summarizing  unpublished work on it.
Some of the results proved in the present article were found by D.W.W. in 2007, and
were stated without proof in that message.
A copy of E.A.'s message has been preserved as part of the OEIS entry \cite{TCS}.


We conjecture that the sequence is finite for any starting value (see \S\ref{SecGraphs}).
The lengths for starting values $1$ through $8$ are
\begin{equation}\label{Eq1.2}
   2137453, 194697747222394, 2, 199900, 19706, 209534289952018960, 15, 198104936410,
\end{equation}
and the corresponding  final terms are respectively
\begin{equation}\label{Eq1.3}
    99999945, 9999999999999918, 36, 9999945, 999945, 9999999999999999936, 936, 9999999999972
\end{equation}
(\seqnum{A330128}, \seqnum{A330129}\footnote{Six-digit numbers prefixed by A refer to entries in the OEIS.}).

The goal of this paper is to try to explain these two bizarre sequences.
We will not be entirely successful:
we cannot predict what the length will be
for a given starting value, but we can at least give a probabilistic
model that explains the huge fluctuations
(\S\ref{SecMines}), and we can explain how the final terms are related to the number of
terms (\S\ref{SecCS}), and what numbers can appear as the final terms (Theorem~\ref{Th1}).

The definition of a comma sequence requires us to always choose the smallest possible successor.
In most cases there is no choice, but occasionally there are two possibilities for the comma-number
(never more), and so two candidates for the comma-successor.
We will refer to the one or two possible
successors as the {\em comma-children} of the previous number.

In the next section we give the formal definitions of a comma
sequence, comma-number, comma-successor, and comma-child in any base $b \ge 2$,
and make a rough estimate for how fast comma-sequences grow.
The definition of comma-number leads naturally to the notion
of the comma transform of a sequence, which is briefly
discussed in \S\ref{SecCT}.  In \S\ref{SecGraphs} we
clarify the distinction between successor and child by defining
the {\em successor graph} $G_s$, where the nodes have out-degree $0$ or $1$, and
the {\em child graph} $G_c$, where some nodes have out-degree $2$.
This section also contains a list of  basic sequences associated
with these graphs,
which will serve as a guide  to \S\ref{SecCST}.

The long \S\ref{SecCST} contains theorems that classify
the {\em landmines}, that is,  numbers without successors,
or equally, without children (Theorem~\ref{Th1}),
numbers with two children (Theorem~\ref{Th2}), numbers
without predecessors (Theorem~\ref{Th3}), and numbers that are not comma-children (Theorem~\ref{Th4}).
In \S\ref{SubSecRoot}, Equation~\eqref{EqRs} traces the numbers in base $10$ back to their roots
in $G_s$.

The final subsection, \S\ref{SubSecKon}, contains the important
Theorem ~\ref{ThKon}, which proves that for
all bases $b \ge 2$, the child graph $G_c$ contains an  infinitely
long path. This  is something we believe is definitely not true for $G_s$, except in base $2$.
In other words, if
we have the freedom to optionally choose the other comma-child
as the next term if there are two comma-children, we can
avoid all the landmines. The proof of this theorem, however,
uses  K\"{o}nig's Infinity Lemma, a version of the Axiom of Choice,
and is not constructive.
We discuss these infinite paths further in \S\ref{SecB3H} and \S\ref{SecB10}.

In \S\ref{SecAP} we study the periodicity of the comma-numbers, and give the analysis that
underlies the computer program that made it possible for us to compute very large
numbers of terms of these sequences.
Section \ref{SecMines} investigates how likely it is for a comma sequence to hit a landmine.
A model is proposed which implies that the expected length of
a comma sequence in base $b$  is asymptotic to $e^{2b}$ for large $b$.
This result was obtained with help from the OEIS itself,
which suggested a match between this question and an apparently unrelated number-theoretic sequence.

Comma sequences in base $2$, which is exceptional, are discussed in \S\ref{SecB2}.
For base $3$ (\S\ref{SecB3}), we can prove that all comma-sequences are finite (Theorem~\ref{ThGR}),
and we  believe we know exactly what the comma-numbers are (Conjecture~\Ref{ConjB3.1}).

The final two sections, \S\ref{SecB3} and \S\ref{SecB10}, are concerned
with infinite paths in the child graphs in bases $3$ and $10$.
For base $3$ we prove that there is a unique infinite path starting at node $1$
 and give an explicit construction for it (Theorem~\ref{Stairway3}).
For base $10$ we have no such conjecture, although we can say with certainty
what the first $10^{84.8}$ terms of the infinite path (or paths) are.

\vspace*{+.1in}

\vspace*{+.1in}

Notation.  The terms of all the sequences discussed here are assumed to be positive integers.
A centered dot ($\cdot$) is a multiplication sign. In base $b$,
a {\em digit} means any number between $0$ and $b-1$, even when $b \ne 10$.
A string of digits with  subscript $b$ signifies a number written in base $b$.
We will usually omit this subscript if the meaning is clear from the context.
This will be especially true in the proofs,
where we will rely on a  sympathetic reader to realize without being told when
 $xy$ (say) is a two-digit number rather than  a product.
 The $n$th term of a sequence is often denoted by $a(n)$.
The leading digit of $n$ will be denoted by $\delta(n)$.
The terms comma-number $(cn(n))$, comma-successor, and comma-child are defined in \S\ref{SecCS}.
The graphs $G_s$ and $G_c$ are defined in \S\ref{SecGraphs}.


\section{Formal definitions of comma sequences; rate of growth}\label{SecCS}
The {\em comma sequence} in base $b \ge 2$ with initial term $s \ge 1$
is the sequence $\{a(n): n \ge 1\}$ defined as follows: $a(1)=s$,
and for $n \ge 2$, if the base-$b$ expansion
of $a(n)$ is $d_1 d_2 \ldots d_m$ then $a(n+1) = e_1 e_2 \ldots e_p$ is the smallest
number such that the base-$b$ number $d_m e_1 ~ (= d_m\cdot b + e_1)$ is equal to the difference
$a(n+1)-a(n)$. If no such number exists, the sequence ends with $a(n)$.
If $a(n+1)$ exists, $d_m e_1$ is the {\em comma-number} associated with the comma between
 $a(n)$ and $a(n+1)$, and $a(n+1)$ is the {\em comma-successor} to $a(n)$.

 One may think of the comma-number as a kind of fattened-up comma, which not only separates the terms
 of the sequence, but says how far apart they are.

 There may be more than one choice for $a(n+1)$, so we define more
 generally  a {\em comma-child}
 of a number $k =  d_1 d_2 \ldots d_m$ to be any number $k' = e_1 e_2 \ldots e_p$
 such that $k' - k = d_m e_1$.  In accordance with the law of
 primogeniture, the first-born child (i.e., the smallest) is the  comma-successor.

Note that
\begin{equation}\label{Eq2.1}
a(n+1) ~=~ a(n) ~+~ d_me_1
\end{equation}
and
\begin{equation}\label{Eq2.2}
1 ~\le~ d_m e_1 ~\le~ b^2-1\,,
\end{equation}
although since $e_1 \neq 0$,  $d_m e_1$ cannot take the values $b, 2\cdot b, \ldots, (b-1)\cdot b$.

It follows that a comma sequence is strictly monotonically increasing (and no terms are zero).
Note also  that if $k' = e_1 e_2 \ldots e_p$  is a successor or child of $k =  d_1 d_2 \ldots d_m$,
then $k$ is uniquely determined by $k'$. For we have $d_m  + e_1 \equiv e_p \pmod{b}$,
which determines $d_m$, and then $k = k' - d_me_1$.
Furthermore, since the comma-number determines the leading digit of a child,
all children have distinct leading digits, and since the comma-number is less than $b^2$,
a positive number can have at most two children.

When the terms are very  large, since
the sequence increases by less than $b^2$ at each step, the leading digit will rarely  change.
A useful rule of thumb is  that
most of the time, a term $d_1 d_2 \ldots d_m$ will
be followed by the comma-number $d_md_1$, and  will have  comma-successor
\begin{equation}\label{Eq2.3}
d_1 d_2 \ldots d_m ~+~ d_m d_1.
\end{equation}
This remark  is the key to extending these sequences to large numbers of terms.
We will return to this point in \S\ref{SecAP}.

An elementary calculation based on \eqref{Eq2.2} shows that the
 average of the {\em possible} values for the comma-number is  $b^2/2$.

  \begin{figure}[!htb]
 \centerline{\includegraphics[angle=0, width=3.4in]{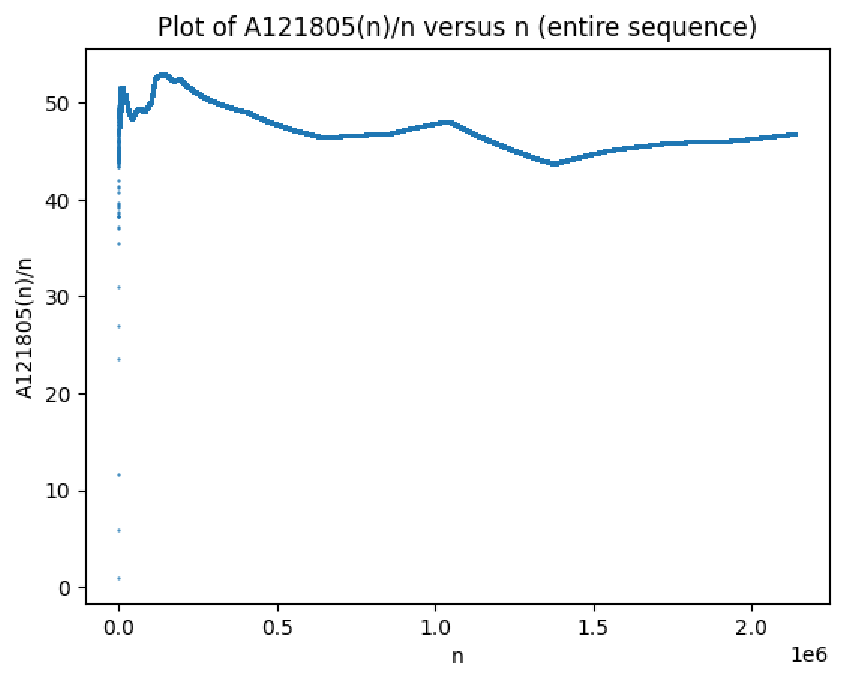}}
 \caption{Plot of $A121805(n)/n$ as $n$ goes from $1$ to $2137453$.
 The final point has ordinate $46.78\ldots$.}
 \label{Fig1}
 \end{figure}

Of course the actual comma-numbers are not random (in fact they are
 extremely regular, as we will see in  \S\ref{SecAP}),
 but nevertheless it seems to be another  reasonable rule of thumb
 that the $n$th term of a comma sequence in base $b$ is roughly $nb^2/2$.
 In the original comma sequence  \eqref{Eq1}, for example, there are $2137452$ comma-numbers,
 whose total is $99999944$, with average value $46.78\ldots$, reasonably close to $b^2/2= 50$.
 However, as Fig.~\ref{Fig1} shows, ``reasonably close'' is about all we can say.

Similarly, the $k$th number in \eqref{Eq1.3} is reasonably close to $50$ times the
$k$th number in \eqref{Eq1.2}.\footnote{To make a fair comparison,  remember that the comma
sequence corresponding to the $k$th number in \eqref{Eq1.2} and \eqref{Eq1.3} has initial term $k$.}


\section{The comma transform}\label{SecCT}
If $\{a(n): n \ge n_0\}$ is a sequence of base-$b$ numbers,
the sequence of comma-numbers corresponding to the commas separating
the $a(n)$ may be called the {\em comma transform} of the sequence.

For example, the comma transform of the nonnegative numbers (in base $10$) is
\begin{equation}\label{Eq367362}
1,  12, 23, 34, 45, 56, 67, 78, 89, 91, 1, 11, 21, 31, \ldots\,.
\end{equation}
The comma transform seems to be a new notion, since this
 sequence (\seqnum{A367362})  was only added to the OEIS in 2023.

The comma sequence as defined in the previous
section can now be redefined as the lexicographically
earliest sequence of positive numbers starting with $s$
 whose sequence of first differences coincides
with its comma transform.


\section{The successor graph and the child graph}\label{SecGraphs}
To clarify the distinction between comma-successor and comma-child, we
define two directed graphs for each base $b \ge 2$.  The {\em successor graph} $G_s$
 has a node for every positive integer
and an edge from $i$ to $j$ if $j$ is the successor to $i$,
and the {\em child graph} $G_c$  has an edge from $i$ to $j$ if $j$ is a child of $i$.
The graph $G_s$ is a subgraph of $G_c$; they are both locally finite.
A portion of the base-$3$ child graph is shown in Fig.~\ref{Fig4}  towards the end of this article.

The following  summary of the properties of these two graphs will also serve as a
guide to the next  section.

By definition, the nodes of $G_s$ have out-degree $0$ or $1$, and as we saw in~\S\ref{SecCS},
each node also has in-degree $0$ or $1$. Thus $G_s$ consists of a number
of isolated nodes and a collection of disjoint paths with edge-length at least $1$.
In base $10$ there are
exactly four isolated nodes, $18, 27, 54,$ and $63$.
Numerical data and the probabilistic arguments given in \S\ref{SecMines} make it likely that
 all the paths in $G_s$ are finite for $b \ge 3$, although we
are only  able to prove this for base $3$.

\begin{conj}\label{ConjGs}
For any base $b \ge 3$ the successor graph $G_s$ does not contain any path of infinite length.
\end{conj}

For the rest of this section we concentrate on base $10$.
The paths that start at nodes
$1$ through $8$ have lengths listed in \eqref{Eq1.2} and end at the nodes listed in \eqref{Eq1.3}.

The nodes in $G_s$ with in-degree $0$ form the infinite sequence
\begin{align}\label{Eq4.1}
&1, 2, 3, 4, 5, 6, 7, 8, 9, 10, 13, 14, 15, 16, 17, 18, 19, 20, 21, 25, 26, 27,  \nonumber \\
&28, 29, 30, 31, 32, 37, 38, 39, 40, 41, 42, 43, 49, 50, 51, 52, 53, 54, 60,  \nonumber \\
&62, 63, 64, 65, 70, 74, 75, 76, 80, 86, 87, 90, 98, 200, 300, 400, 500, 600, \ldots
\end{align}
(Theorem~\ref{Th3}, \seqnum{A367600}), and those with in-degree $1$
are given in the complementary sequence
\begin{equation}\label{Eq4.2}
11, 12, 22, 23, 24, 33, 34, 35, 36, 44, 45, 46, 47, 48, 55, 56, 57, 58, 59, \ldots
\end{equation}
(\seqnum{A367340}).
The nodes in $G_s$ (or, equally, in $G_c$) with out-degree $0$ are
 \begin{align}\label{Eq4.3}
& 18, 27, 36, 45, 54, 63, 72, 81, 918, 927, 936, 945, 954, 963, 972, 981, 9918, 9927, \nonumber \\
& 9936, 9945, 9954, 9963, 9972, 9981, 99918,  99927, 99936, 99945, 99954, 99963, \nonumber \\
&   99972, 99981, 999918, 999927, 999936, 999945, 999954, 999963, 999972, \ldots
\end{align}
(Theorem~\ref{Th1}, \seqnum{A367341}), and those with out-degree $1$
are given in the complementary sequence
\begin{equation}\label{Eq4.4}
1, 2, 3, 4, 5, 6, 7, 8, 9, 10, 11, 12, 13, 14, 15, 16, 17, 19, 20, 21, 22, 23,  \ldots
\end{equation}
(\seqnum{A367615}).
Understanding the terms of sequence \eqref{Eq4.3}
and its generalization to other bases is the key to the whole paper. These are the  landmines,
and if a sequence reaches one of these nodes it ends there.
They play a central role in the following sections.
Note that there are just eight landmines with each of two, three, four, $\ldots$ digits.
They are rare, and  they only occur  in the final $100$ terms
between two powers of $10$.

The comma-successor to $n \ge 1$ (if it exists, or $-1$ if it doesn't) is
given by
\begin{equation}\label{Eq4.5}
12, 24, 36, 48, 61, 73, 85, 97, 100, 11, 23, 35, 47, 59, 72, 84, 96, -1, 110, \ldots\,,
\end{equation}
(\seqnum{A367338}).

There does not seem to be any better way of describing the graph $G_s$ other than by giving
the  lists shown in \eqref{Eq1.2}, \eqref{Eq1.3}, and \eqref{Eq4.1} to \eqref{Eq4.5}.
But  $G_c$ has a richer structure.

The nodes of the child graph $G_c$, like those of $G_s$,  have in-degree $0$ or $1$, only now
there are only finitely many of in-degree $0$. There are exactly $50$ of these,
and they are all one- or two-digit numbers:
\begin{align}\label{Eq4.6}
1, &2, 3, 4, 5, 6, 7, 8, 9, 10, 13, 14, 15, 16, 17, 18, 19, 20, 21, 25, 26, 27, 28, 29, 30, 31, 32,   \nonumber \\
 &37, 38, 39, 40, 41, 42, 43, 49, 50, 51, 52, 53, 54, 62, 63, 64, 65, 74, 75, 76, 86, 87, 98
\end{align}
(Theorem~\ref{Th4}, \seqnum{A367611}).  This means that $G_c$  consists of $50$ connected
components, each of which is a directed rooted tree. As we will see in \S\ref{SecB10}
the tree rooted at $20$ is infinite and the other $49$ are finite.

The nodes of in-degree $1$ form the complementary sequence to \seqnum{A367611}, \seqnum{A367612}.

The nodes of out-degree $0$ are the same as for $G_s$ (see \eqref{Eq4.6}).
As we saw  in \S\ref{SecCS}, a node in $G_c$ can have at most two children.
The nodes with out-degree $2$ are
\begin{equation}\label{Eq4.7}
14, 33, 52, 71, 118, 227, 336, 445, 554, 663, 772, 881, 1918, 2927, 3936, \ldots
\end{equation}
(Theorem~\ref{Th2}, \seqnum{A367346}). These are the
{\em branch-points}, the  nodes where the paths in the  child graph $G_c$ fork.
We will discuss them further in  \S\ref{SecB3H} and \S\ref{SecB10}.
When seeking an infinite path in $G_c$ we must
navigate through these branch-points so as to avoid the landmines.
The nodes of out-degree $1$ in $G_c$ comprise \seqnum{A367613}: these are the nodes not in  \seqnum{A367341} or \seqnum{A367346}.


\section{The successor and predecessor theorems}\label{SecCST}
By definition the comma-successor is independent of any earlier
terms in the sequence.
So we can try to gain insight by studying
 the sequence whose $n$th
term is the comma-successor to $n$ in the given base, or
 $-1$ if $n$ has no successor.
In base $10$ this comma-successor sequence
was mentioned in the previous section (see \eqref{Eq4.5}).
In base $2$ the comma-successor sequence is simply $4,3,6,5,8,7,10,9,\ldots$,
which is essentially \seqnum{A103889} (see \S\ref{SecB2}).
In \S\ref{SecB3} we give  an explicit but complicated formula for the base-$3$ comma numbers: see
Conjecture~\ref{ConjB3.1} and \eqref{EqB3.2}.
No analogs of Conjecture~\ref{ConjB3.1} are presently known for higher bases.


\subsection{Numbers without successors}\label{SubSecNS}

\begin{theorem}\label{Th1}
In base $b \ge 2$, the numbers
with no comma-successors (or, equally, with no comma-children)
are precisely the numbers whose representation in base $b$
has the form
\begin{equation}\label{EqCS}
b-1~ b-1\, \ldots\, b-1\, x\, y, {\mathrm{~~with~}} i \ge 0 {\mathrm{~digits~}} b-1,
\, 1 \le x \le b-2,  \, y = b-1-x.
\end{equation}
\end{theorem}

\noindent{\bf Remarks.}
(i) The number \eqref{EqCS} is equal to
\begin{equation}\label{Eq3.3}
b^2 (b^i-1) + (b-1)(x+1)\,,
\end{equation}
where $i \ge 0$, $1 \le x \le b-2$.

(ii) In base $b=2$ no numbers satisfy \eqref{EqCS}, and indeed  as we will see in \S\ref{SecB2}, every
number has a successor. So for the following  proof we may assume $b \ge 3$.
For base $10$ the numbers were listed in \eqref{Eq4.3}.

\begin{proof}
We assume $b \ge 3$. Let $L$ denote the list
of numbers $k$ of the form \eqref{EqCS}. We will prove (1) that if $k \in L$ then $k$ has no
comma-successor, and (2) if a number $k$ has no comma-successor then $k \in L$.
It is easy to see that every single-digit number has a successor, so we will only consider
numbers with two or more digits.

To simplify the discussion we assume $b=10$. The proof for general $b$ is essentially the same,
after replacing $10$ by $b$, $9$ by $b-1$, and $8$ by $b-2$.

Part (1).  Consider  $k = 99\ldots 9 x y \in L$, with $i \ge 0$ $9$s, $1 \le x \le 8$,  and  $x+y=9$.
Suppose, seeking a contradiction, that $k$ has a comma-successor $k'$.
If the leading digit $\delta(k') = 9$, the comma-number is $y9$,
but then $xy + y\cdot 9 \ge 100$, implying $\delta(k') = 1$, a contradiction.
On the other hand, if $\delta(k')=1$, the comma-number is $y1$,
and now $xy + y1 < 100$, implying $\delta(k') = 9$, again a contradiction.

Part (2). Suppose $k$ has no comma-successor.
We consider  two sub-cases: (2a) $k$ has at least three digits,
and (2b) $k$ has two digits.

(2a) Suppose $k = fsxy$, where $f, x, y$ are single digits,
and $s$ is a number with $i \ge 0$ digits.
If $fs$ is not of the form $99\ldots9$, then adding a two-digit number to $k$
will not change its leading digit, and we can simply
take $k' = k+yf$ to be $k$'s successor. Since $k$ does not have a successor, we must
assume that $k = 99\ldots 9xy$. If that $k$ has a successor $k'$, let its leading digit be $g$,
so that $k' = 99\ldots9xy+yg$. If $g=9$, there was no overflow from the last two
digits, so $xy+y9 <100$, and then $k'$ would be a legitimate successor to $k$.
To exclude this, we must have $xy+y9\ge 100$, or, in other words,
\begin{equation}\label{Eq3.5}
10\cdot (x+y) + y + 9 \ge 100\,.
\end{equation}
On the other hand, if $g = 1$, then we could take
$k' = 99\ldots9xy + y1$ unless $xy+y1 < 100$.
So we also need
\begin{equation}\label{Eq3.6}
10\cdot (x+y) + y  \le 98\,.
\end{equation}
Equations \eqref{Eq3.5} and \eqref{Eq3.6} together
imply that $xy$ must be one of $18$, $27, \ldots 81$, as claimed.

(2b) It remains to consider the two-digit case $k = xy$ with $1 \le x \le 9$,
$0 \le y \le 9$. If $y=0$ it is easy to see that $k$ does have a successor,
so we assume $1 \le y \le 9$, and suppose that $k$ has a successor
\begin{equation}\label{Eq3.7}
k' = xy + yi, \mathrm{~with~} 1 \le i \le 9\,.
\end{equation}

If $x+y < 9$, $k'$ will certainly exist. To see this, we divide the range of values of $i$ into a
`low' region, defined by $1 \le i \le 9-y$, where  $\delta(xy+yi) = x+y$,
and a `high' region, defined by $10-y \le i \le 9$, where $\delta(xy+yi) = x+y+1$.
If $x+y$ is in the low-$i$ region we take $i=x+y$ to define $k'$,
and if $x+y+1$ is in the high-$i$ region (that is,
if $9-y \le x+y \le 8$) we take $i=x+y+1$.
The two regions cover all $x+y$ from $1$ to $8$, so $k'$ always exists.
(If $x+y = 9-y$, there are two solutions for $k'$. We will return to this in Theorem~\ref{Th2}.)

On the other hand, if $x+y > 9$,  when we form the sum in \eqref{Eq3.7},
$k'$ becomes a three-digit number with $\delta(k')=1$, and we can take $i=1$ to get a successor.

Finally, if $x+y=9$, there is no solution for $i$. In the low-$i$ region
$\delta(k') = 9$, but $1 \le i \le 9-y$ excludes $9$,
and in the high-$i$ region $\delta(k')=1$, but
$10-y \le i \le 9$ excludes $1$.
So $x+y=9$ is the only case where there is no successor. This completes the proof of (2b).
\end{proof}


\subsection{Numbers with two children}\label{SubSecTwin}

\begin{theorem}\label{Th2}
(a) The one- and two-digit numbers in base $b \ge 2$ with two
children are precisely the numbers
$wx_b = w\cdot b+x$ with $0 \le w \le b-3$,
$1 \le x \le \left\lfloor \frac{b-1}{2} \right\rfloor$
and $w = b-1-2x$.
(b) The numbers with  three or more digits and two children  are the numbers
\begin{equation}\label{Eq5.6}
d ~ b-1 ~ b-1 \, \ldots \, b-1 ~ d ~ b-1-d ~=~ d ~ (b-1)^i \, d~(b-1-d)\,,
\end{equation}
where $1 \le d \le b-2$ and  there are $i \ge 0$ digits $b-1$ following the initial $d$.
\end{theorem}

\noindent{\bf Remarks.}  In base $b=2$ the conditions are vacuous and no numbers have two children.
In bases $b = 3$, $4$,  $5$, and $6$ these numbers are, written in base $b$,
with a semicolon between the (a) and (b) terms:
\begin{align}\label{Eq5.7}
3: ~& 1; 111, 1211,12211, 122211, 1222211, \dots \nonumber \\
4: ~& 11; 112, 221, 1312, 2321, 13312, 23321, \ldots \nonumber \\
5: ~& 2, 21; 113, 222, 331, 1413, 2422, 3431, \ldots \nonumber  \\
6: ~& 12, 31; 114, 223, 332, 441, 1514, 2523, \ldots \nonumber \\
\end{align}
For base $10$ see \eqref{Eq4.7}.
\begin{proof}
We omit the straightforward proof of (a). (b) If $k$ is of the form \eqref{Eq5.6}
then its two children are $d~ (b-1)^{i+2}$ and $(d+1)~0^{i+2}$.
Conversely, suppose $k$ has two children, $k'$ and $k''$,  and at least three digits,
say $k = fswx$, where $f,w,x$ are single digits and $s$ has $i \ge 0$ digits.
Suppose the children are $k' = fswx + xy'$ and $k'' = fswx + xy''$.
The leading digits $y', y''$ of $k', k''$ are distinct, so we can assume
$1 \le y' < y'' \le b-1$.

The only way this can happen is if $s = (b-1)^i$ for $i \ge 0$,
$y'=f$, and $y'' = f+1$. This implies $1 \le f \le b-2$.
Also we have $(w+x)\cdot b+x+f < b^2$ and $(w+x)\cdot b +x+f+1 \ge b^2$,
implying that
\begin{equation}\label{Eq5.8}
(w+x)\cdot b + x + f ~=~b^2-1\,.
\end{equation}
 But $1 \le x+f \le 2b-3$, so \eqref{Eq5.8} implies $w+x \le b-1$.
 However, $w+x \le b-2$ and $b \ge 2$ lead to a contradiction in \eqref{Eq5.8},
 so $w+x=b-1$, and then  \eqref{Eq5.8} implies $x+f= b-1$, so $f=w$.
 But $x=0, f=b-1$ and $x=b-1, f=0$ are impossible, so $1 \le x \le b-2$.
 We conclude that $fswx$ has the form specified in \eqref{Eq5.6}.
 \end{proof}

\subsection{Numbers without predecessors}\label{SubSecNPD}
In this subsection we discuss the numbers that are not
the comma-successor of any number, that is,
the nodes with in-degree $0$ in $G_s$ (Theorem~\ref{Th3}),
and those that are not the comma-child of any number, that is,
those with in-degree $0$ in $G_c$ (Theorem~\ref{Th4}).

The following useful lemma constructs the parent of a child, if it exists. We omit the elementary proof.

\begin{lemma}\label{Lemma1}
In base $b \ge 2$, if $n$ is a comma-child of $k$, then
\begin{equation}\label{Eq5.L}
k = n - x\cdot b - f,  \text{~where~} f = \delta(n) \text{~and~} x = (n-f) \bmod{b}.
\end{equation}
If $n - x\cdot b - f$ is negative, then $n$ is not a comma-child.
\end{lemma}

\begin{theorem}\label{Th3}
In base $b \ge 3$, the only numbers $n \ge b^2-1$ that are not
comma-successors are the numbers
\begin{equation}\label{Eq5.9}
c \cdot b^i \text{~with~} i \ge 2 \text{~and~} 2 \le c \le b-1.
\end{equation}
\end{theorem}

\begin{proof}
If $n \ge b^2-1$ is not of the form  \eqref{Eq5.9}
then it is easy to verify that first, $n$
is  a child of $k = n - f -x \cdot b$ where $f = \delta(n)$
and $x = (n-f) \pmod{b}$, and second, that $k$ has no smaller child than $n$,
implying that $n$ is the successor to $k$.
(ii) On the other hand, if $n = c \cdot b^i$ ($i \ge 2, 2 \le c \le b-1$),
then although $n$ is a child of
$$
k = c \cdot b^i - b^2 + c \cdot (b-1) = (c \cdot b^{i-2} - 1) \cdot b^2 + (c-1) \cdot b + (b-c)\,,
$$
$k$ also has a smaller child, $n-1$, and so $n$ is not the successor to $k$.
\end{proof}

 \noindent{\bf Remark.}
In base $2$, the only numbers without predecessors are $1$ and $2$.
 In base $10$ the full list of numbers that are not successors is given in \eqref{Eq4.1}: there
 are $54$ terms less than $99$.  For a general base $b \ge 3$, it appears that there
 are exactly $(b^2+b-2)/2$ numbers less than $b^2-1$ that are not successors.
 These numbers appear to depend in a complicated way on the value of $b$, and we
 have not attempted to classify them.


 \begin{theorem}\label{Th4}
In base $b \ge 2$, the only numbers $n$ that are not comma-children are in
the range $1 \le n \le b^2-1$.
\end{theorem}
\begin{proof}
If $n \ge b^2$ then $k$ given by \eqref{Eq5.L} is positive, and so $n$ has a smaller parent.
\end{proof}
For base $10$ these are the $50$ numbers listed in \eqref{Eq4.6}.


\subsection{The ancestors of n}\label{SubSecRoot}

In base $10$, the sequence $\{R_s(n), n \ge 1\}$ giving the most remote ancestor of $n$ in the
graph $G_s$ is
\begin{align}\label{EqRs}
& 1, 2, 3, 4, 5, 6, 7, 8, 9, 10, 10, 1, 13, 14, 15, 16, 17, 18, 19, 20, 21, 20, \nonumber \\
& 10, 2, 25, 26, 27, 28, 29, 30, 31, 32, 30, 21, 1, 3, 37, 38, 39, 40, 41, 42, \nonumber \\
& 43, 40, 31, 20, 13, 4, 49, 50, 51, 52, 53, 54, 50, 41, 32, 10, 14, 60, 5, 62, \ldots \nonumber \\
\end{align}
(\seqnum{A367366}).
From \eqref{Eq1.2} we know this sequence contains
$2137453$ $1$s, $194697747222394$ $2$s,  $2$ $3$s, and so on.
As mentioned in \S\ref{SecGraphs}, we
 believe that for any base $b > 2$ every path in $G_s$ is finite.

For base $10$, the  sequence $\{R_c(n), n \ge 1\}$  (\seqnum{A367617}) giving the most remote ancestor
of $n$ in the  graph $G_c$ is the same as \eqref{EqRs} for the first $59$ terms.
The first difference is at $n=60$, where because $14$ has two children, $59$ and
$60$, we have $R_c(60) = 14$ whereas $R_s(60) = 60$.


\subsection{The child graph contains infinitely-long paths}\label{SubSecKon}

The major difference between the successor graph $G_s$  and the child graph
$G_c$ is in the following theorem.


 \begin{theorem}\label{ThKon}
 For any base $b \ge 2$, the child graph $G_c$ contains
 an infinite path.
\end{theorem}
\begin{proof}
We know from Theorem~\ref{Th4}  that every node
in $G_c$ belongs to a tree with root-node less than $b^2$.
By the Infinite Pigeonhole Principle one of these trees contains infinitely many nodes,
and by K\H{o}nig's Infinity Lemma \cite[p. 2]{CST}, \cite[p. 233]{ZAC} that tree contains an infinite branch.
\end{proof}

We will say more about these infinite paths in Sections~\ref{SecB3} and \ref{SecB10}.


\section{Predicting comma-numbers; the algorithm}\label{SecAP}

In this section we explain how we are able to compute comma sequences for the huge numbers
of terms shown in \eqref{Eq1.2}.
The reason is that the comma-numbers are extremely regular.
As mentioned at the end of \S\ref{SecCS}, if the leading digit does
not change,  there is a simple formula \eqref{Eq2.3} for the comma-number
and the next term. This causes the comma-numbers to form
simple periodic sequences for long stretches.

As an illustration, consider the original comma sequence \seqnum{A121805}, \eqref{Eq1}.
The leading digit changes as we go from $a(1942) = 99987$
to $a(1943)=100058$,
and stays constant at $1$ until the next
change, which is from $a(4114) = 199959$
to $a(4115)=200051$.
In that long stretch of $2171$ consecutive terms all beginning with $1$, the comma sequence
consists of  $217$ repetitions of the ten-term arithmetic progression 
$\{81, 91, 1, 11, 21, 31, 41, 51, 61, 71\}$, plus one additional copy
of $81$.\footnote{This data can be found in
OEIS entries \seqnum{A121805} and \seqnum{A366487}.} The sum of the ten terms
is $460$, so from $a(1943)$ to $a(4114)$ the sequence increased by $460\cdot 217 + 81 = 99901$,
reaching $a(4115)=200051$.
This calculation replaced  $2171$ individual calculations.

To see why these arithmetic progressions occur, suppose we know a term
$a(n) = k$ in a base-$b$ comma sequence. To avoid irregularities at the start
of the sequence, suppose $k = fswx_b$, where $f, w,x$ are single base-$b$ digits and $s$
has $i \ge 0$ digits. As long as
\begin{equation}\label{Eq6.1}
k \le (f+1)\cdot(b^{i+2}-b^2)\,,
\end{equation}
the leading digit of the next term $a(n+1) = k'$
will also be equal to $f$,
and so $k' = k + x\cdot b + f$.
This means that the comma sequence from $k$
onwards satisfies the recurrence
\begin{align}\label{Eq6.2}
k(0) & = k\,, \nonumber \\
k(j+1) & = k(j) + (k(j) \bmod{b})\cdot b + f \text{~for~} j \ge 0\,,
\end{align}
as long as \eqref{Eq6.1} continues to hold for $k(j+1)$.
Now $k(j) \bmod{b}= (x+j\cdot f) \bmod{b}$, so we obtain
\begin{equation}\label{Eq6.3}
k(n) ~=~ k ~+~ \sum_{j=o}^{n-1} ((x+j\cdot f) \bmod{b}) + n\cdot f\,,
\end{equation}
for $n \ge 0$.
The right-hand side of  \eqref{Eq6.3}  only depends on the value of $n \bmod{b}$,
so by taking $n=mb$ we can jump ahead $mb$ steps
in the sequence, obtaining
\begin{equation}\label{Eq6.4}
k(m\cdot b) ~=~ k ~+~ m \cdot \left(b \cdot f + \sum_{j=0}^{b-1} ((x+j\cdot f) \bmod{b})\right)\,,
\end{equation}
provided $k(m\cdot b)$ satisfies \eqref{Eq6.1}.

Equation~\eqref{Eq6.4} is the basis for our computer
algorithm, which uses it to skip over long stretches of the sequence
whenever possible. Implementations in various computer languages
may be found in \seqnum{A121805}.


\section{Chance of hitting a landmine}\label{SecMines}
In this section we give two random models for the comma sequences, which
partially explain why the huge fluctuations in \eqref{Eq1.2}  are not surprising.

Let us fix the base $b \ge 2$.
From a distance, a comma sequence may look like a kangaroo moving along a very long path in small jumps, each jump being of length between 1 and $b^2-1$. There are landmines on the path,
and if the kangaroo lands on one, its journey ends.  The path has infinitely many squares, numbered $1, 2, 3, \ldots$. The landmines are concentrated in bunches. On the squares numbered
with $k$-digit base-$b$ numbers,
that is, the squares in the range $b^{k-1}$ to $b^k-1$,
the mines are concentrated in the final $b^2$ squares, and in that range there
are $b-2$ mines (Theorem~\ref{Th1}).

So our first, naive, model says that the chance of the kangaroo
landing on a landmine is $(b-2)/b^2$, which is (correctly) zero if $b=2$ and is $8/100$
in base $10$. Naively, then, we could say that in base $10$ the chance
of not  getting from one power of $10$ to the next is $8/100$,  or $1/12.5$,
and so the expected life of the kangaroo
is about $10^{12.5}$. This is much higher
than what we have observed, and certainly this model is too crude.
For one thing, the kangaroo does not pick a square at random from
the final $b^2$ squares, it progresses through them in jumps of
average size $b^2/2$.

For a more realistic model, we  start a kangaroo (in base $10$)
at one of the $100$ squares   $99\ldots9xy =  9^m xy$,
with $m \ge 2$ and $0 \le x \le 9$, $0 \le y \le 9$,
and follow it,  to see if it gets  to safety at or beyond
$9^{m+1}00$. It turns out that $88$ succeed (independent of
the choice of $m$), so the chance of hitting a landmine is $12/100$,
which suggests an expected life of $10^{100/12} = 10^{8.33}$.
This is consistent with the data in \eqref{Eq1.2} and its $10000$-term extension in \seqnum{A330128}.

We repeated this computation for all bases $b$  from $2$ to $100$,
taking the kangaroo's starting point to be in the range $(b-1)^m00$ to $(b-1)^{m+2}$, where $m \ge 2$.
The number, $D(b)$, say,  that died in base $b \ge 2$ is given by the $b$th term of
the sequence
\begin{equation}\label{Eq7.1}
0, 1, 2, 4, 5, 7, 8, 11, 12, 14, 16, 18, 20, 23, 24, 26, 29, 31, 33, 36, 38, 40, 42,  \ldots
\end{equation}
This sequence was not in \cite{OEIS} (although it is now \seqnum{A368364}),
but to our surprise its first differences  appeared to
agree\footnote{Apart from minor differences in the first two terms.}  with \seqnum{A136107},
which has several equivalent definitions,
one of which is that it  is the number of ways to write
a number as the difference of two triangular numbers. It also has
 an explicit generating function.

There is no obvious connection between the two
problems, but an hundred-term agreement is unlikely to be a coincidence,
so we state what it implies about \eqref{Eq7.1}  as a conjecture.\footnote{January 20 2024: Robert Dougherty-Bliss and Natalya Ter-Saakov \cite{RDB} report  that they have proved this conjecture.}

\begin{conj}
\label{Darien}
The number $D(b)$  of comma sequences in base $b \ge 2$ that start in the range
$(b-1)^m00$ to $(b-1)^{m+2}$, $m \ge 2$, but do not reach $b0^{m+2}$,
is given by the coefficient of $t^b$ in the expansion of
\begin{equation}\label{Eq7.2}
\frac{1}{1-t} \left( \sum_{b=1}^{\infty}  \frac{t^{b(b+3)/2}}{(1-t^b)} - t^2 \right)
~=~ t^3 + 2t^4 + 4 t^5 + 5 t^6 + 7t^7 + \ldots
\end{equation}
\end{conj}

Supposing Conjecture~\ref{Darien} to be correct,
we can estimate the coefficients $D(b)$ for large $b$.
We are grateful to V\'{a}clav Kot\v{e}\v{s}ovec (personal communication)
for the following analysis.

$D(b)$ is essentially the $b$th partial sum of \seqnum{A136107},
and from the OEIS entry for
that sequence we see that $\seqnum{A136107}(b) = \alpha(b) - \beta(b)$,
where $\alpha(b) = \seqnum{A001227}(b)$ is the number of odd divisors of $b$,
and $\beta(b) = \seqnum{A010054}(b)$ is $1$ if $b$ is a triangular number, and
is otherwise $0$. The contribution to the partial sum from $\beta(b)$
is $O(\sqrt{b})$ and can be ignored.
$\alpha(b)$ has Dirichlet generating function $\zeta(s)^2 (1-\frac{1}{2^s})$,
so from Perron's formula \cite[p.\ 245]{Apostol}, \cite[p.\ 217]{TEN} the partial sum
$\alpha(1)+\ldots+\alpha(b)$ is, for large $b$,  asymptotic to
the residue at $s=1$ of
\begin{equation}\label{Eq7.3}
\zeta(s)^2 \left(1 -\frac{1}{2^s}\right) \frac{b^s}{s}\,.
\end{equation}
Computer algebra systems such as Maple and Mathematica can compute this residue,
which after simplification is 
\begin{equation}\label{Eq7.4}
b \left( \frac{\log (2b)}{2} + \gamma - \frac{1}{2} \right)\,,
\end{equation}
where $\gamma$ is the Euler-Mascheroni constant.
Not surprisingly, this expression \eqref{Eq7.4} for the partial sums of the
number-of-odd-divisors function (with $b$ changed to $n$) resembles the classical formula
\begin{equation}\label{Eq7.5}
\tau(1) + \ldots + \tau(n) ~=~ n\,(\log n + 2 \gamma -1) ~+~ O(\sqrt{n})
\end{equation}
for the partial sums of the number of divisors function $\tau(n)$
\cite[Theorem~320]{HW},  \cite[\S3.1, Eq. (2)]{Apostol}, \cite[\S3.2, Th.\ 3.2]{TEN}.

For our purpose it is enough to use the leading term of \eqref{Eq7.4},
which gives $D(b) \sim  \frac{1}{2} b(\log 2b)$,
so the chance that the kangaroo  does not reach
the next power of $b$ is $D(b)/b^2 \sim  \frac{\log 2b}{2b}$.
The expected length of the comma sequence (assuming Conjecture~\ref{Darien}) is therefore
$\sim b^{2b/\log 2b} ~\sim~ e^{2b}$
as $b \to \infty$. For $b=10$ this is $10^{8.69}$, reasonably close
to the value $10^{8.33}$ obtained above.


\section{Base 2}\label{SecB2}
Base $2$ is exceptional: there are just two comma sequences, and both are infinite.
If we start at $1$ the sequence (in binary) is
$1, 100, 101,1000, 1001, 1100, 1101,10000, 10001, \ldots$
or $1,4,5,8,9,12,13,16,17, \ldots$
if written in decimal (essentially \seqnum{A042948}, the numbers $4k$ and $4k+1$).
If we start at $2$ the sequence is
$10, 11, 110, 111, 1010, 1011,1110, 1111, \ldots$
or $2,3,6,7,10,$ $11,$ $14,$ $15,\ldots$
in decimal (\seqnum{A042964}, the numbers  $4k+2$ and $4k+3$).
 Every number belongs to one of these two sequences.


\section{Base 3 and the finiteness of comma sequences}\label{SecB3}
For bases  greater than $2$ we cannot come close to a full analysis.
But for base $3$, at least, we can prove
that all comma sequences are finite, and we believe we
have an explicit formula for the comma-numbers $cn(n)$.

\begin{theorem}\label{ThGR}
In base $3$ the successor graph $G_s$ does not contain
an infinite path.
\end{theorem}
\begin{proof}
We will prove that a comma sequence starting at any $x \ge 1$ always hits a
landmine, which, in base $3$,
has the form $22\dots211$, i.e., $3^h-5$. For simplicity,  in the
following discussion
we assume that $x$ has at least four digits (so $x\ge 3^3$),
since the smaller cases can be resolved computationally.

In general, a sequence starting at $3^{k-1}\le x < 3^{k}$  will either
hit the landmine $3^{k}-5$ or contain a term $y$ in the range $3^{k}\le
y \le 3^{k}+7$,
since comma-numbers do not exceed $3^2-1=8$.

Considering the eight  ways a sequence can approach the boundary $3^{k}$ from
below, we can further observe that the sequence starting at $x$, if it
does not terminate earlier, will hit one of the numbers $3^k+u$ with
$u\in\{0, 2, 3, 6\}$.

Again, a sequence containing  $3^k+u$
will either end at $3^{k+1}-5$ or continue to term $3^{k+1}+v$, with
$u,v\in\{0,2,3,6\}$.
By making  the relation between $k$, $u$, and $v$ explicit, we will prove
that all sequences end.

We have already seen in \S\ref{SecAP} that comma-numbers are
periodic in nature.
For base $3$, the following table shows the  comma-numbers corresponding to the
terms of a sequence
starting at $w\cdot 3^k +z$, with $k\ge 3$, $w=1,2$, and $z=0,1,2$,
  as long as the first digit of the terms remains unchanged:

\begin{center}\begin{tabular}{l|l}
start & comma-numbers\\\hline\\[-10pt]
$3^k$   & 1, 4, 7, 1, 4, 7, 1, 4, 7,\dots\\
$3^k+1$ & 4, 7, 1, 4, 7, 1, 4, 7, 1,\dots\\
$3^k+2$ & 7, 1, 4, 7, 1, 4, 7, 1, 4,\dots\\
$2\cdot 3^k$ & 2, 8, 5, 2, 8, 5, 2, 8, 5,\dots\\
$2\cdot 3^k+1$ & 5, 2, 8, 5, 2, 8, 5, 2, 8,\dots\\
$2\cdot 3^k+2$ & 8, 5, 2, 8, 5, 2, 8, 5, 2,\dots
\end{tabular}
\end{center}

So, in general, as long as the leading digit of the terms remains the
same, every three terms a sequence increases by $1+4+7=12$ (leading
digit 1) or by $2+8+5=15$ (leading digit 2).

For $k>2$, we observe that $3^k \pmod{12}$ has a period of length 2, and
$3^k \pmod{15}$ has a period of length 4. After a straightforward but
tedious case-by-case analysis, taking into account the temporary disruption in the
pattern when the leading digit of the terms changes, we obtain the
following table, which  summarizes the evolution of a sequence between two
powers of 3:

\begin{center}
\begin{tabular}{l|l}
from & arrives at \\\hline\\[-10pt]
$3^{4h}$ & $3^{4h+1}+2$\\
$3^{4h}+2$ & \bfseries ends \\
$3^{4h}+3$ & $3^{4h+1}+0$\\
$3^{4h}+6$ & $3^{4h+1}+6$\\ \\

$3^{4h+1}$ & \bfseries ends\\
$3^{4h+1}+2$ & $3^{4h+2}+3$  \\
$3^{4h+1}+3$ & $3^{4h+2}+2$\\
$3^{4h+1}+6$ & $3^{4h+2}+0$\\
\end{tabular}\qquad%
\begin{tabular}{l|l}
from & arrives at \\\hline\\[-10pt]
$3^{4h+2}$ & \bfseries ends \\
$3^{4h+2}+2$ & $3^{4h+3}+6$ \\
$3^{4h+2}+3$ & $3^{4h+3}+2$\\
$3^{4h+2}+6$ & $3^{4h+3}+3$\\ \\

$3^{4h+3}$ & $3^{4h+4}$\\
$3^{4h+3}+2$ & \bfseries ends \\
$3^{4h+3}+3$ & $3^{4h+4}+3$\\
$3^{4h+3}+6$ & $3^{4h+4}+6$\\
\end{tabular}
\end{center}

The corresponding transition graph is shown in Figure \ref{Fig2}, and
 makes it clear that all sequences terminate, and that there is
an upper bound on the number of powers of 3 a sequence starting at $x$ can
pass before ending. This completes the proof.
\end{proof}

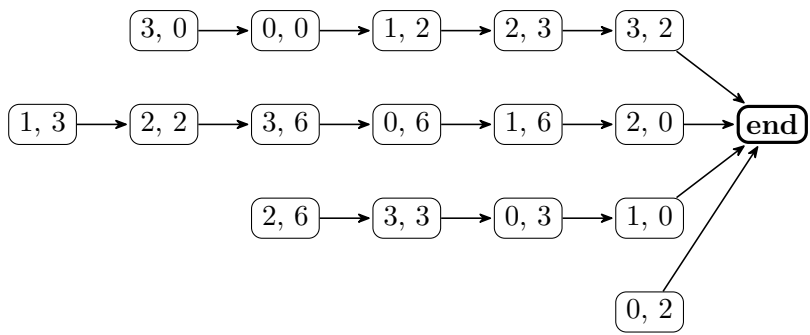
\begin{figure}[!ht]
\begin{tikzpicture}[node distance=0.7cm,scale=1]
\tikzstyle{every node}=[draw,rectangle,rounded corners,inner
sep=3pt,minimum width=9mm]
\node[very thick] (end) {\bfseries end};
\node (n20) [left=of end] {2, 0};
\node (n16) [left=of n20] {1, 6};
\node (n06) [left=of n16] {0, 6};
\node (n36) [left=of n06] {3, 6};
\node (n22) [left=of n36] {2, 2};
\node (n13) [left=of n22] {1, 3};
\node (n32) [above=of n20] {3, 2};
\node (n23) [left=of n32] {2, 3};
\node (n12) [left=of n23] {1, 2};
\node (n00) [left=of n12] {0, 0};
\node (n30) [left=of n00] {3, 0};
\node (n10) [below=of n20] {1, 0};
\node (n03) [left=of n10] {0, 3};
\node (n33) [left=of n03] {3, 3};
\node (n26) [left=of n33] {2, 6};
\node (n02) [below=of n10] {0, 2};
\path[->,>={Stealth[round]},shorten >=1pt,semithick]%
(n02) edge (end)
(n26) edge (n33) (n33) edge (n03) (n03) edge (n10) (n10) edge (end)
(n13) edge (n22) (n22) edge (n36) (n36) edge (n06) (n06) edge (n16)
(n16) edge (n20) (n20) edge (end)
(n30) edge (n00) (n00) edge (n12) (n12) edge (n23) (n23) edge (n32)
(n32) edge (end);
\end{tikzpicture}
\caption{Transitions between numbers near powers of $3$.
A node labeled $(s,t)$ represents numbers of the form
$3^{4k+s}+t$, and ``end'' represents the landmines at  $3^k-5$.
An edge from $(s,t)$ to $(w,z)$ with $w = (s+1 \bmod 4)$ means that a sequence
containing the term $3^{4k+s}+t$ also contains the term
$3^{4k+s+1}+z$.}
\label{Fig2}\centering
\end{figure}
For example, consider a sequence starting at $x$, with $3^{12}\le x<
3^{13}$.
It either stops at $3^{13}-5$ or reaches one of the numbers $3^{13}+\{0,
2, 3, 6\}$.

If it reaches $3^{13}+3$, since $13\equiv1 \pmod{4}$, according to
Figure \ref{Fig2}, it will also contain the terms $3^{14}+2$,
$3^{15}+6$, $3^{16}+6$, $3^{17}+6$, $3^{18}$, and will end at $3^{19}-5$.

\vspace*{+.1in}

\noindent{\bf Remark.}
In principle it should be possible to use similar arguments to prove
that the successor graph in any base $b \ge 3$ does not contain an infinite path,
although the details will become increasingly complicated, and may require computer assistance.

\vspace*{+.1in}


The values of the comma number $cn(n)$  for
with $n \ge 1$ (writing both $n$ and $cn(n)$ in ternary, but using
$-1$ to indicate that no comma-successor exists) are:
 $$
   \begin{array}{rrrrrrrrrrrrrr}\label{ttt}
   n: & 1 & 2 & 10 & 11 & 12 & 20 & 21 & 22 & 100 & 101 & 102 & 110 & \ldots \\
   cn(n): & 11 & 21 & 1 & -1 & 21 & 2 & 11 & 21 & 1 & 11 & 22 & 1   & \ldots
   \end{array}
$$
(\seqnum{A367609}).
Examination of a much larger table suggests the following, which
 is a more precise version of the rule-of-thumb in \eqref{Eq2.3}.
Although we are confident this is correct, we state it as a conjecture, since we will not give a formal proof.

\begin{conj}\label{ConjB3.1}  In base $b=3$,
the comma-number $cn(n)$ associated with $n = d_1d_2 \ldots d_m$
is $d_m d_1$ with the
exceptions shown in Table~\ref{Tab1}.

\begin{table}[!ht]
  \caption{For these values of $n$ the true comma-number $cn(n)$  and the
  correction to be added to $d_md_1$ are:}\label{Tab1}
$$
\begin{array}{|c|c|c|}
\hline
n & \mbox{cn(n)} & \mbox{correction } \\
\hline
1 2^i 1~~ (i \ge 1)               & 12  & +1 \\
1 2^i j 2 ~~(i \ge 0, j=0,1,2) & 22 & +1 \\
2 \text{~or~} 22                   & 21  & +1 \text{~or~} -1 \\
2^i 1 ~~(i \ge 0)                   & 11   & -1 \\
2^i j 2 ~~(i \ge 0, j=0,1,2)    & 21   & -1 \\
2^i 11~~ (i \ge 0)                 & \text{(does~not~exist)} & \text{(none)}  \\
\hline
\end{array}
$$
\end{table}
\end{conj}
 The first row of the table, for example, implies that
if $n = 1221$, for which $d_md_1 = 11$, the true comma-number is $12$, so
we must add $1$ to $d_md_1$.

The conjecture implies that any comma sequence in base $3$ satisfies the recurrence
\begin{equation}\label{EqB3.2}
a(n+1) = a(n) + (d_md_1)_3\,,
\end{equation}
where $a(n) = d_1 \ldots d_m$,
except that a correction must be added to $(d_m d_1)_3$ (or the sequence must be terminated),
when $a(n)$ is one of the  exceptions listed in Table~\ref{Tab1}.
This is a long way from Fibonacci's recurrence,
but---see the opening sentence---that was only to be expected.


\section{The unique infinite path in the base-3 child graph}\label{SecB3H}
We know from
Theorem~\ref{ThKon} that we {\em can} avoid all the landmines if we are permitted
to  pick different children at the branch-points. In this section we give a simple construction
for an infinite path in the base $3$ case, and show it is essentially unique.

 \begin{figure}[!htb]
 \centerline{\includegraphics[angle=0, width=4in]{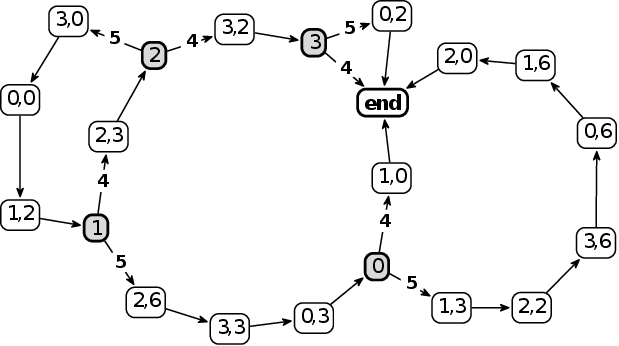}}
 \caption{The graph in Fig.~\ref{Fig2} enlarged to include the branch-points, which are the shaded nodes.
 The unique infinite path follows the loop on the left of the graph.}
 \label{Fig3}
 \end{figure}

We know from Theorem~\ref{Th2} (see also \eqref{Eq5.7}) that the nodes in $G_c$
where there are branch-points are  $1$ and the  nodes
\begin{equation}\label{Eq8.1}
1 2^i 11_3 = 2\cdot 3^{i+2} - 5 \text{ for } i \ge 0\,.
\end{equation}
The nodes that have no successors are $2^i 11_3 = 3^{i+2}-5$ for $i \ge 0$ (Theorem~\ref{Th1}).

\begin{thm}\label{Stairway3}
There is a unique infinite path in the base-$3$ child graph $G_c$ that starts at $1$.
A rule for constructing it is that when branch-points are reached, the path should alternately choose the
lower and the higher alternatives.
\end{thm}

\begin{proof}
We use the machinery developed in the previous section.
Figure~\ref{Fig3} shows  the graph in Fig.~\ref{Fig2} enlarged
to include the branch-points. For $h \in \{0,1,2,3\}$, the branch points
of the form $2 \cdot 3^{4k+h}-5$ for some $k$ are represented by the shaded node
labeled $h$ in Fig.~\ref{Fig3}. The two outgoing edges from the shaded nodes
are labeled with the comma numbers, which are always $4$ and $5$.
The infinite path in the graph repeats the loop on the
left of Fig.~\ref{Fig3}.  Considered purely as a graph, the infinite path is clearly unique,
since all other loops contain the  ``end'' node , and it is also clear that the infinite path alternately chooses $4$ and $5$ at the branch-points.

Suppose a sequence contains the term
$3^{4 \cdot 3 + 2}+3$, represented by the node $(2,3)$.
The sequence will then continue to the branch-point
 $t=2\cdot 3^{4\cdot3+2}-5$.
From there, choosing the higher term $t+5$, the sequence will reach $3^{4\cdot3+3}$, i.e., node (3,0).
The sequence proceeds to the terms $3^{4\cdot4}$ (node $(0,0)$),
$3^{4\cdot 4+1}+2$ (node $(1,2)$),
and then to the branch-point $2\cdot3^{4\cdot4+1}-5$ (the shaded node $1$). Following the lower branch marked with `4', we arrive at $3^{4\cdot 4+2}+3$, which is of the same form as our starting point, and the process repeats.

Since completing the loop the exponent increases by only $4$, we can conclude that there is only one infinite path: every infinite path must contain a term of the form $3^{4k+2}+3$, but once a path includes this term, it will include all subsequent terms of the same form.
\end{proof}

The infinite sequence, starting at $1$, begins
\begin{equation}\label{Eq9.3}
1, 5, 12, 13, 18, 20, 27, 28, 32, 39, 40, 44, 51, 52, 57, 59, 67, 72, 74, 81, \ldots\,.
\end{equation}
(\seqnum{A367621}).
The branch points actually encountered  are $1$, $111$, $1 2^311$, $12^411$, $12^711$, $1 2^811, \ldots$,
alternating nodes $1 2^{4j-1}11$ (where we take the lower branch)
 and $12^{4j}11$ (where we take the higher branch).

The beginning of the path through $G_c$ is shown in Figure~\ref{Fig4}.
The thick arrows indicate the correct paths to take out of the branch-points, while
the long thin arrows are bad choices that lead to landmines  where those
paths end.  Short vertical arrows denote long strings of edges without branches.

\begin{figure}[!ht]
\centerline{\includegraphics[width=5.25in]{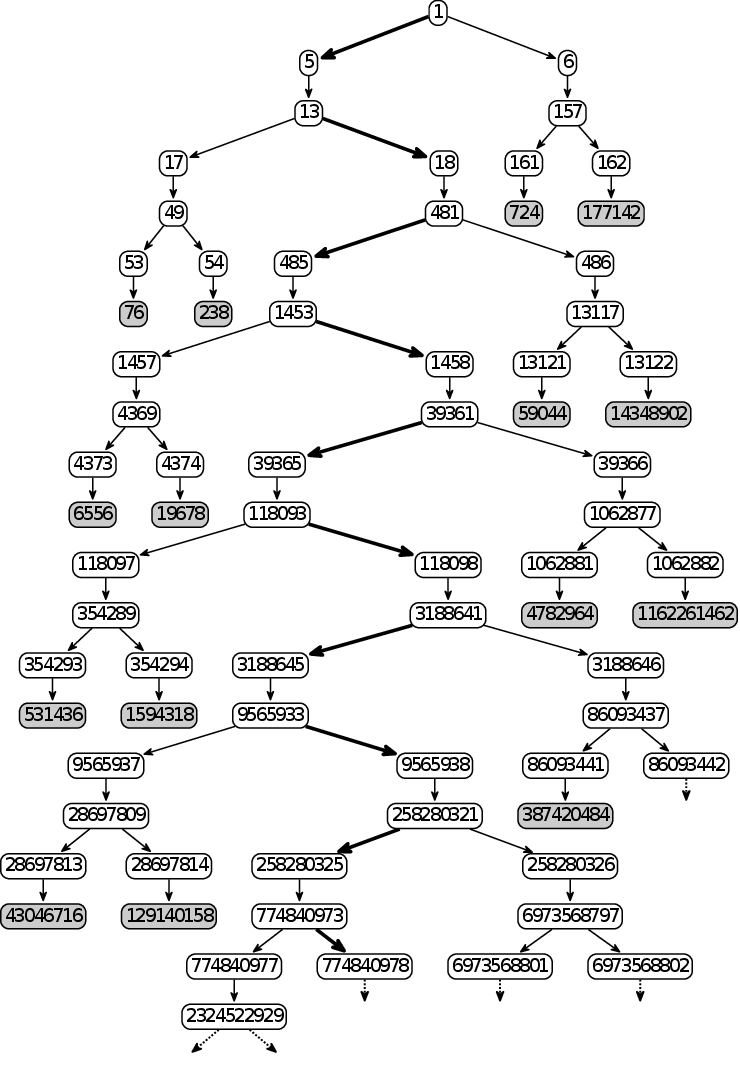}}
\caption{A portion of the base-$3$ child graph $G_c$.
The thick lines are the  start of the unique infinite path  (\seqnum{A367621}).
The shaded nodes are the landmines, which must be avoided.
 Short vertical arrows  represent long unbranched paths.}
\label{Fig4}
\end{figure}

\section{An infinite path in the base-10 child graph}\label{SecB10}
Finally, we return to where we began, with base $10$.
We know from Theorems~\ref{Th4} and \ref{ThKon} that
the child graph contains an infinite path that begins at a node in the range $1$ to $99$.
We also know (see \eqref{Eq4.6}) that these paths start at one of  $50$ root-nodes.
By computer, we followed all the paths in these $50$ trees
until all but one had terminated. The sole survivor
turned out to be  one of the paths with root $20$.
The longest rival had root $30$, and persisted for $10^{365} - 82$ terms before reaching a landmine.

We can now focus on the surviving  start, and  define \seqnum{A367620}
to be the lexicographically earliest infinite sequence in $G_c$. We know it has initial terms
\begin{equation}\label{10.1}
20, 22, 46, 107, 178, 260, 262, 284, 327, 401, 415, 469, 564, 610, 616, 682, \dots\,.
\end{equation}
Its first branch-point  is at A367620(412987860) = 19999999918.

Starting from the beginning of this sequence, we have followed all its continuations
through the first $69$  branch-points.
By $30$ branch-points
there were three survivors left, one of which made these choices at the branch-points:
\begin{equation}\label{EqC1}
0 0 1 1 1 0 0 1 1 1 0 1 1 0 0 0 0 1 1 0 0 1 0 1 1 1 1 1 1 0
\end{equation}
(where $0$ = down, $1$ = up), and reached the branch-point  $29^{84}27_{10}$
at step about $10^{84.8}$.
The other two candidates failed at a later stage, so we can be confident
of the first $10^{84.8}$ terms of \seqnum{A367620}. However, unlike the ternary case,
there is no obvious pattern to \eqref{EqC1}.
At $69$ branch-points, there are several candidates still in the running,
all of which begin with \eqref{EqC1}.
We know that one or more of them will extend to infinity; it would be nice to know more.

\section{Acknowledgments}
We thank  Ivan N. Ianakiev  for a helpful comment,
and V\'{a}clav Kot\v{e}\v{s}ovec for the asymptotic estimate in \eqref{Eq7.3}.
Thanks also to Robert Dougherty-Bliss and Natalya Ter-Saakov for comments on the manuscript.

\medskip

\noindent MSC2020: 11B37, 11B75

\end{document}